\documentclass[reqno,12pt]{amsart}

\usepackage{color}

\theoremstyle{plain}
\newtheorem{thm}{Theorem}[section]
\newtheorem{lemma}[thm]{Lemma} 
\newtheorem{prop}[thm]{Proposition}

\theoremstyle{remark}

\newtheorem{remark}[thm]{Remark}

\theoremstyle{definition}
\newtheorem{defi}[thm]{Definition}
\newtheorem{example}[thm]{Example}

\newtheorem{ques}[thm]{Question}

\usepackage{amscd,amssymb,euscript}
\usepackage{graphicx}
\usepackage{enumerate}
\usepackage[initials]{amsrefs}
\usepackage[all]{xy}
\usepackage{comment}

\newcount\theTime
\newcount\theHour
\newcount\theMinute
\newcount\theMinuteTens
\newcount\theScratch
\theTime=\number\time
\theHour=\theTime
\divide\theHour by 60
\theScratch=\theHour
\multiply\theScratch by 60
\theMinute=\theTime
\advance\theMinute by -\theScratch
\theMinuteTens=\theMinute
\divide\theMinuteTens by 10
\theScratch=\theMinuteTens
\multiply\theScratch by 10
\advance\theMinute by -\theScratch

\def\today{{\number\day\space
 \ifcase\month\or
  January\or February\or March\or April\or May\or June\or
  July\or August\or September\or October\or November\or December\fi
 \space\number\year}}

\newcommand\at{{\tilde a}}

\newcommand\Bc{{\mathcal{B}}}
\newcommand\Cpx{{\mathbb C}}
\newcommand\diag{\operatorname{diag}}

\newcommand\Fb{{\mathbb F}}
\newcommand\HEu{{\EuScript H}}                   
\renewcommand{\i}{\text{\rm i}}

\newcommand{\im}{\text{\rm Im}}
\newcommand\ImagPart{{\mathrm{Im}\;}}

\newcommand\lambdat{{\tilde\lambda}}

\newcommand\Mcal{{\mathcal{M}}}
\newcommand\Nats{{\mathbb N}}

\newcommand\Reals{{\mathbb R}}

\newcommand\RealPart{{\mathrm{Re}\;}}
\newcommand{\re}{\text{\rm Re}}

\newcommand\supp{\operatorname{supp}}

\newcommand\UT{\operatorname{UT}}
\newcommand\UTTM{\operatorname{UTTM}}
\newcommand\Vc{{\mathcal{V}}}
\newcommand\tr{{\mathrm{tr}}}
\newcommand\zt{{\tilde z}}

\begin{document}

\title[Upper triangular matrices]{Upper triangular Toeplitz matrices and real parts of quasinilpotent operators}

\author[Dykema]{Ken Dykema$^{1}$}
\address{K.D., Department of Mathematics, Texas A\&M University,
College Station, TX 77843-3368, USA}
\email{kdykema@math.tamu.edu}
\author[Fang]{Junsheng Fang$^{2}$}
\address{J.F.,  School of Mathematical Sciences,  Dalian University of Technology,
Dalian, China}
\email{junshengfang@gmail.com}
\author[Skripka]{Anna Skripka$^3$}
\address{A.S., Department of Mathematics and Statistics, MSC01 1115,
	University of New Mexico, Albuquerque, NM 87131, USA}
\email{skripka@math.unm.edu}
\thanks{\footnotesize ${}^1$Research supported in part by NSF grants DMS--0901220 and DMS--1202660.
${}^2$Partially supported by the Fundamental Research Funds for the Central Universities of China and NSFC(11071027).
${}^3$Research supported in part by NSF grants DMS-0900870 and DMS--1249186}


\date{10 October, 2012}

\begin{abstract}
We show that every self--adjoint matrix $B$ of trace $0$
can be realized as  $B=T+T^*$ for a nilpotent matrix $T$
with $\|T\|\le K\|B\|$,
for a constant $K$ that is independent of matrix size.
More particularly, if
$D$ is a diagonal, self--adjoint $n\times n$ matrix of trace $0$, then there is a unitary matrix $V=XU_n$,
where $X$ is an $n\times n$ permutation matrix and $U_n$ is the $n\times n$ Fourier matrix,
such that the upper triangular part, $T$,
of the conjugate $V^*DV$ of $D$
satisfies $\|T\|\le K\|D\|$.
This matrix $T$ is a strictly upper triangular Toeplitz matrix
such that $T+T^*=V^*DV$.
We apply this and related results to give partial answers to questions about real parts of quasinilpotent elements
in finite von Neumann algebras.
\end{abstract}

\maketitle

\section{Introduction}
\label{sec1}

It is well known and easy to show (by induction) that every self--adjoint matrix whose trace vanishes
is unitarily equivalent to a matrix having zero diagonal;
therefore, it is equal to the real part of a nilpotent operator.

Recall that an element $z$ of a Banach algebra is quasinilpotent if its spectrum is $\{0\}$, and that this
is equivalent to $\lim_{n\to\infty}\|z^n\|^{1/n}=0$.
Fillmore, Fong and Sourour showed~\cite{FFS79} that a self-adjoint operator $T$
on an infinite dimensional separable Hilbert space can be realized
as the real part $(Z+Z^*)/2$ of a quasinilpotent operator $Z$ if and only if
$0$ is in the convex hull of the essential spectrum of $T$.

Since
each quasinilpotent element of a II$_1$--factor has trace equal to zero
(by, for example, Proposition~4 of~\cite{MW79})
the following question seems natural:
\begin{ques}\label{ques:typeII1}
If $\Mcal$ is a II$_1$--factor with trace $\tau$ and if $a=a^*\in\Mcal$ has $\tau(a)=0$,
must there be a quasinilpotent operator $z\in\Mcal$ with $a=z+z^*$?
\end{ques}

Analogously, the following question is also natural:
\begin{ques}\label{ques:typeI}
If $\Mcal$ is a finite type I von Neumann algebra, and if $a=a^*\in\Mcal$ has center--valued trace equal to zero,
must there be a quasinilpotent operator $z\in\Mcal$ with $a=z+z^*$?
\end{ques}

An answer to Question~\ref{ques:typeI} will, necessarily, and an answer to Question~\ref{ques:typeII1} will, most likely,
involve a quantitative understanding of the problem in matrix algebras.
The main result of this paper (Theorem~\ref{thm:partial}) is a step in this direction.

Our interest in quasinilpotent operators in II$_1$--factors is partially motivated by the paper~\cite{HS09} of Haagerup
and Schultz.
In it, they show that every element of a II$_1$--factor whose Brown measure is not concentrated at a single point,
has a nontrivial hyperinvariant subspace.
Since the support of the Brown measure is contained in the spectrum, quasinilpotent operators are
examples of those to which the Haagerup--Schultz theorem does not apply and, indeed,
the hyperinvariant subspace problem remains open for quasinilpotent operators in II$_1$--factors.

The following result is a straightforward consequence of Theorem~8.1 of~\cite{HS09}.
\begin{thm}[\cite{HS09}]
For any element $T$ of a finite von Neumann algebra $\Mcal\subseteq\Bc(\HEu)$,
\begin{equation}\label{eq:HS}
A:=\text{s.o.t.--}\lim_{n\to\infty}((T^*)^nT^n)^{1/2n}
\end{equation}
exists, and $\supp(\mu_T)=\{0\}$ if and only if $A=0$.
\end{thm}
The notation in~\eqref{eq:HS} is for the limit in strong operator topology on $\Bc(\HEu)$.
This result characterizes those operators to which the Haagerup--Schultz result on existence of hyperinvariant subspaces
does not apply, in terms that resemble
a characterization of quasinilpotency.
This motivates the following nomenclature.
\begin{defi}
Let $\HEu$ be a Hilbert space and let $T\in\Bc(\HEu)$.
We say $T$ is {\em s.o.t.--quasinilpotent} if
\[
\text{s.o.t.--}\lim_{n\to\infty}((T^*)^nT^n)^{1/2n}=0.
\]
\end{defi}
Clearly, every quasinilpotent element is s.o.t.--quasinilpotent, and the hyperinvariant subspace problem for elements
of II$_1$--factors is reduced to the question for s.o.t.--quasinilpotent operators in II$_1$--factors.
Furthermore, the analogues of Questions~\ref{ques:typeII1} and~\ref{ques:typeI} where ``s.o.t.--quasinilpotent''
replaces ``quasinilpotent'' are interesting, and we will answer positively the second of these.

Before we describe our main results, here some
interesting examples related to Question~\ref{ques:typeII1}.
\begin{example}
Let $\{x_1,x_2\}$ be free semicircular operators that generate the free group factor $L(\Fb_2)$.
Then $x_1$ and $x_2$ are real parts of quasinilpotent operators in $L(\Fb_2)$.
Indeed, $x_i/2$ is the real part of a
copy of the quasinilpotent DT--operator in $L(\Fb_2)$, by results of~\cite{DH04}.
\end{example}

\begin{example}
G.\ Tucci~\cite{T08}
found a family $(A_\alpha)_{0<\alpha<1}$ of quasinilpotent elements of the hyperfinite II$_1$--factor $R$,
each generating $R$ as a von Neumann algebra.
He showed for each $\alpha$, $\RealPart(A_\alpha)$ has the same moments as $\ImagPart(A_\alpha)$ and he found
a combinatorial formula for them.
He showed that
each $\RealPart(A_\alpha)$ generates a diffuse subalgebra of a Cartan masa in $R$, which is for some values of $\alpha$
all of the Cartan masa and for other values is a proper subalgebra of it.
\end{example}

Now we describe our main results.
It is straightforward to see (the details can be found in Section~\ref{sec2})
that if a diagonal matrix $D=\diag(\lambda_1,\cdots,\lambda_n)$ has zero trace,
then the conjugate $B=U_n^*DU_n$ of this matrix by the $n\times n$ Fourier matrix $U_n=\frac1{\sqrt n}(\omega_n^{(j-1)(k-1)})_{1\le j,k\le n}$,
where $\omega_n=e^{2\pi\i/n}$, is a Toeplitz matrix (meaning the $(i,j)$th entry depends only on $i-j$),
has all zeros on the diagonal, and the upper triangular part of it, which we will call $T_\lambda$, satisfies
$T_\lambda+T_\lambda^*=B$.
Here $\lambda$ denotes the sequence $(\lambda_1,\ldots,\lambda_n)$ and we have $\|D\|=\|\lambda\|_\infty:=\max_j|\lambda_j|$.
Note that $T_\lambda$ is in fact an upper triangular Toeplitz matrix, and is nilpotent.
A key issue is: how large is the norm of $T_\lambda$ compared to the norm of $D$?

The matrix $T_\lambda$ is the image of $B$ under the upper triangular truncation operator.
The asymptotic behaviour of the norm of this upper triangular truncation operator on the $n\times n$ matrices was
determined by Angelos, Cowen and Narayan in~\cite{ACN92} to be $\frac1\pi\log(n)+O(1)$ as $n\to\infty$
(see Example 4.1 of~\cite{D88} and~\cite{KP70} for earlier results).
Our main result (Theorem~\ref{thm:partial}) is that there is a constant $K$ such that for every
natural number $n$ and every finite
real sequence $\lambda=(\lambda_1,\ldots,\lambda_n)$
that sums to zero, there is a rearrangement $\lambdat$ of $\lambda$
such that $\|T_\lambdat\|\le K\|\lambda\|_\infty$.
A value for the constant $K<1.78$ (though not, to our knowledge, the best possible value)
and the rearrangement $\lambdat$ are found explicitly.
The only requirement on the rearrangement
is that the partial sums of the rearranged sequence do not exceed $\|\lambda\|_\infty$ in absolute value.

We observe that rearrangement is necessary by
making the estimate (Proposition~\ref{prop:lowerbd}) that when $\lambda=(1,\ldots,1,-1,\ldots,-1)$ of
length $2n$ sums to zero,
then we have $\|T_\lambda\|\ge\frac1\pi\log(2n)+C$ for a constant $C$, independent of $n$.
In fact, the same asymptotic lower bound estimate, but for some different upper triangular Toeplitz matrices,
was obtained by Angelos, Cowen and Narayan~\cite{ACN92}.

We also prove a slightly different rearrangement result of a similar nature (Proposition~\ref{prop:Dm2}),
for use in taking inductive limits.

In Section~\ref{sec:vN}, we apply our main theorem to give some results in type I von Neumann algebras
related to Question~\ref{ques:typeI}
and also draw some consequences in II$_1$--factors.

In Section~\ref{sec:indlim}, we apply the related rearrangement result in an inductive limit
to prove results about II$_1$--factors.
Finally, we ask a further specific question.

{\bf Acknowledgement:}  The authors thank Paul Skoufranis for helpful comments about an earlier version of the paper.

\section{Upper triangular Toeplitz matrices}
\label{sec2}

For $n\in\Nats$ and $p\in\{0,1,\ldots,n\}$, let $M_n$ denote the set of $n\times n$ matrices with complex values and let $\UT_n^{(p)}$ denote the set of matrices $x\in M_n$ that have zero entries everywhere below the diagonal and on the first $p$ diagonals on and above the main diagonal.
That is, $x=(x_{ij})_{1\le i,j\le n}$ belongs to $\UT_n^{(p)}$ if and only if $x_{ij}=0$ whenever $j<i+p$.
So a matrix is strictly upper triangular if and only if it belongs to $\UT_n^{(1)}$.

An $n\times n$ matrix $X=(x_{ij})_{1\le i,j\le n}\in M_n$ is said to be a {\em Toeplitz matrix}
if $x_{ij}$ depends only on $i-j$.
We let $\UTTM_n^{(p)}$ be the set of all Toeplitz matrices that belong to $\UT_n^{(p)}$.
Every $T\in\UTTM_n^{(1)}$ is of the form
\begin{equation}\label{eq:T}
T=\left(\begin{matrix}
0&t_1&t_2&\cdots&t_{n-1} \\
&0&t_1&\ddots&\vdots \\
&&\ddots&\ddots&t_2 \\
&&&0&t_1 \\
&&&&0
\end{matrix}\right)
\end{equation}
and is nilpotent.
Moreover, $\UTTM_n^{(0)}$ is a commutative algebra.

We now describe in more detail the matrices $T_\lambda$ mentioned in the introduction.
Let $\omega_n=e^{2\pi \i/n}$.
Recall that then $\sum_{j=0}^{n-1}\omega_n^{dj}=0$ whenever $d$ is an integer
that is not divisible by $n$ and, consequently,
$f_1,f_2,\ldots,f_n$ is an orthonormal basis for $\Cpx^n$, where
\[
f_k=\frac1{\sqrt n}\sum_{j=1}^n\omega_n^{(k-1)(j-1)}e_j.
\]

Let $\Vc_n$ be the real vector space consisting of all real sequences
$\lambda=(\lambda_1,\ldots,\lambda_n)$
such that $\sum_{j=1}^n\lambda_j=0$.

For $\lambda\in\Vc_n$,
consider the matrix $D=\diag(\lambda_1,\ldots,\lambda_n)$ with respect to the standard basis in $\Cpx^n$,
and let us write it as a matrix, $B$, with respect to the basis $f_1,\ldots,f_n$.
We have
\begin{equation}\label{eq:Dff}
\langle Df_k,f_l\rangle
=\frac1n\bigg\langle\sum_{p=1}^n\lambda_p\omega_n^{(k-1)(p-1)}e_p,\sum_{q=1}^n\omega_n^{(l-1)(q-1)}e_q\bigg\rangle
=\frac1n\sum_{p=1}^n\lambda_p\omega_n^{(k-l)(p-1)}.
\end{equation}
Then the change--of--basis matrix whose columns are $f_1,\ldots,f_n$ is the $n\times n$ Fourier matrix,
$U_n=\frac1{\sqrt n}(\omega_n^{(j-1)(k-1)})_{1\le j,k\le n}$,
and we have
\begin{align}\label{eq:number}
B=U_n^*DU_n=\left(\begin{matrix}
0&t_1&t_2&&\cdots&t_{n-1} \\
\overline{t_1}&0&t_1&\ddots&&\vdots \\
\overline{t_2}&\overline{t_1}&0&\ddots \\
\vdots&\ddots&\ddots&\ddots&\ddots&\vdots \\
\vdots&&&\ddots&0&t_1 \\
\overline{t_{n-1}}&\cdots&&\cdots&\overline{t_1}&0
\end{matrix}\right)
\end{align}
is a Toeplitz matrix, where
\begin{equation}\label{eq:td}
t_d=t_d(\lambda)=\frac1n\sum_{p=1}^n\lambda_p\omega_n^{d(p-1)}.
\end{equation}
Let $T_\lambda\in\UTTM_n^{(1)}$ be the upper triangular part of $B$.
By construction, $T_\lambda+T_\lambda^*$ is a self--adjoint $n\times n$ matrix
with eigenvalues $\lambda_1,\ldots,\lambda_n$.
Thus, the map $\Vc_n\ni\lambda\mapsto T_\lambda\in\UTTM_n^{(1)}$ is linear and injective.
\begin{remark}\label{rem:Tlambdas}
From~\eqref{eq:td} we see
\begin{equation}\label{eq:tnd}
t_{n-d}=\overline{t_d},\qquad(1\le d\le n-1).
\end{equation}
Considering dimensions, we see that the map $\lambda\to T_\lambda$ is a linear ismomorphism from $\Vc_n$
onto the set of complex upper triangular Toeplitz matrices
of the form~\eqref{eq:T}
for which~\eqref{eq:tnd} holds.
\end{remark}

For $\lambda\in\Vc_n$, let $\|\lambda\|_\infty=\max_j|\lambda_j|$.
We regard such sequences as maps from $\{1,\ldots,n\}$ to $\Reals$ and, thus,
for $\sigma\in S_n$, i.e., $\sigma$ a permutation of $\{1,\ldots,n\}$,
$\lambda\circ\sigma\in\Vc_n$ denotes the sequence $(\lambda_{\sigma(1)},\ldots,\lambda_{\sigma(n)})$.
As described in the introduction, we will find a constant $K$, independent of $n$,
such that for every $\lambda\in\Vc_n$, there is $\sigma\in S_n$ such that
\begin{align}
\label{eq:thebound}
\|T_{\lambda\circ\sigma}\|\le K\|\lambda\|_\infty=K\|B\|.
\end{align}

We will require some elementary lemmas.
The next lemma is a simple observation about a known series expansion of the cotangent function.
\begin{lemma}\label{lemma2}
\begin{enumerate}[(i)]
\item
For $x\in (0,1)$,
\[
\cot(\pi x)=\frac{1}{\pi x}-\frac{1}{\pi}\sum_{k=1}^\infty \frac{2x}{k^2-x^2}\,.
\]
\item
The function \[f_1(x)=\frac{1}{\pi}\sum_{k=1}^\infty \frac{2x}{k^2-x^2}\] increases on the set $[0,\frac12]$ to the maximum value $\frac{4}{\pi}\sum_{k=1}^\infty\frac{1}{4k^2-1}=\frac{2}{\pi}$.
\end{enumerate}
\end{lemma}

\begin{lemma}\label{lem:Hn}
Let $n\in\Nats$ and let
\[
H_n=\frac1n\begin{pmatrix}
0 &\i&\i&\cdots& \hspace*{0.8em} \i \\
-\i&0&\i& & \hspace*{0.8em} \vdots \\
-\i&-\i&0&\ddots& \hspace*{0.8em} \vdots \\
\vdots&&\ddots&\ddots& \hspace*{0.8em} \i \\[1ex]
-\i &\cdots &\cdots & -\i & \hspace*{0.8em} 0 \hspace*{0.8em}
\end{pmatrix}
\]
be the $n\times n$ self-adjoint matrix having zeros on the diagonal and all entries above the diagonal equal to $\i:=\sqrt{-1}$.
Then an orthonormal list of eigenvectors of $H_n$ is $(v_k)_{k=0}^{n-1}$,
and the associated eigenvalues are $(\mu_k)_{k=0}^{n-1}$, where
if $n$ is odd, then
\begin{align}
v_k&=\frac1{\sqrt n}(1,-\omega_n^k,(-\omega_n^k)^2,\ldots,(-\omega_n^k)^{n-1})^t, \notag \\
\mu_k&=\frac{\i}n\sum_{j=1}^{n-1}(-\omega_n^{k})^j=\frac{1}{n}\tan\left(\frac{\pi k}{n}\right) \label{eq:mukOdd}
\end{align}
and we have $\mu_0=0$ and $\mu_k=-\mu_{n-k}$ if $1\le k\le n-1$,
while if $n$ is even, then
\begin{align}
v_k&=\frac1{\sqrt n}(1,\omega_n^k\omega_{2n},(\omega_n^k\omega_{2n})^2,
 \ldots,(\omega_n^k\omega_{2n})^{n-1})^t, \notag \\
\mu_k&=\frac{\i}n\sum_{j=1}^{n-1}(\omega_n^k\omega_{2n})^j=-\frac 1n\cot\left(\frac{\pi(2k+1)}{2n}\right) \label{eq:mukEven}
\end{align}
and we have $\mu_{n-1-k}=-\mu_k$ for all $0\le k\le n-1$.
Note that in all cases, we have $v_k=W_n^kv_0$ for all $0\le k\le n-1$, where
\[
W_n=\diag(1,\omega_n,\omega_n^2,\ldots,\omega_n^{n-1}).
\]
\end{lemma}
\begin{proof}
One calculates $H_nv_k$ to be the vector whose $p$th entry, for $1\le p\le n$, is
\begin{equation}\label{eq:Hnvk}
\begin{cases}
\frac{\i}{n\sqrt n}\left(-\sum_{j=0}^{p-2}(-\omega_n^k)^j+\sum_{j=p}^{n-1}(-\omega_n^k)^j\right),&n\text{ odd} \\
\frac{\i}{n\sqrt n}\left(-\sum_{j=0}^{p-2}(\omega_n^k\omega_{2n})^j+\sum_{j=p}^{n-1}(\omega_n^k\omega_{2n})^j\right),
&n\text{ even,}
\end{cases}
\end{equation}
where the sum $\sum_{j=0}^{p-2}$ is taken to be zero if $p=1$,  as is the sum $\sum_{j=p}^{n-1}$ if $p=n$.
The quantity~\eqref{eq:Hnvk} equals $\frac1{\sqrt n}(-\omega_n^k)^{p-1}$ or,
respectively, $\frac1{\sqrt n}(\omega_n^k\omega_{2n})^{p-1}$
times $\mu_k$, for $n$ odd and, respectively, even.
This shows that $v_k$ is an eigenvector with eigenvalue $\mu_k$.

By standard properties of geometric progressions and trigonometry, we derive for $n$ odd
\begin{equation*}
\mu_k=\frac{\i}{n}\cdot\frac{1-\omega_n^k}{1+\omega_n^k}=\frac2n\cdot\frac{\im(\omega_n^k)}{|1+\omega_n^k|^2}=\frac{1}{n}\tan\left(\frac{\pi k}{n}\right)
\end{equation*}
and for $n$ even
\begin{equation*}
\mu_k=\frac{\i}{n}\cdot\frac{\omega_{2n}^{2k+1}+1}{1-\omega_{2n}^{2k+1}}
=\frac{-2}{n}\cdot\frac{\im(\omega_{2n}^{2k+1})}{|1-\omega_{2n}^{2k+1}|^2}
=-\frac 1n\cot\left(\frac{\pi(2k+1)}{2n}\right).
\end{equation*}
All other assertions follow easily.
\end{proof}

\begin{remark}\label{rem:Hn}
From Lemma~\ref{lem:Hn} we get
\[
H_n=\sum_{k=1}^n\mu_kW_n^kQ_n(W_n^*)^k,
\]
where $Q_n$ is the rank--one projection onto the span of the vector $v_0$,
and the $n$ rank--one projections $W_n^kQ_n(W_n^*)^k$ for $1\le k\le n$ are pairwise orthogonal.
(Note: we let $\mu_n=\mu_0$ and $v_n=v_0$, while of course $W_n^n$ is the identity matrix.)
\end{remark}

The following facts follow directly from the formulas~\eqref{eq:mukOdd} and~\eqref{eq:mukEven}.
\begin{remark}\label{rem:muk}
If $n$ is odd, then the sequence $(\mu_k)_{k=1}^n$ is of the form
\begin{equation*}
(a_1,a_2,\ldots,a_{(n-1)/2},-a_{(n-1)/2},\ldots,-a_2,-a_1,0),
\end{equation*}
where
\[
0<a_1<a_2<\cdots<a_{(n-1)/2}<\frac2\pi,
\]
while if $n$ is even then $(\mu_k)_{k=0}^{n-1}$ is of the form
\begin{equation*}
(-b_1,-b_2,\ldots,-b_{n/2},b_{n/2},\ldots,b_2,b_1),
\end{equation*}
where
\[
\frac2\pi>b_1>b_2>\cdots>b_{n/2}>0.
\]
\end{remark}

The next lemma is just the well known scheme behind Dirichlet's test.
We include the proof for convenience.
\begin{lemma}\label{lem:DIrichlet}
Let $n\in\Nats$ and suppose $a_1,\ldots,a_n,b_1,\ldots,b_n\in\Reals$ where
the sequence $a_1,\ldots,a_n$ monotone and $b_1+\cdots+b_n=0$.
Let
\[
M=\max_{1\le k\le n}\left|\sum_{j=1}^k b_j\right|.
\]
Then
\[
\left|\sum_{j=1}^n a_jb_j\right|\le M|a_n-a_1|.
\]
\end{lemma}
\begin{proof}
Let $B_k=\sum_{j=1}^kb_j$, with $B_0=0$.
By hypothesis, $B_n=0$.
Then
\[
S:=\sum_{j=1}^na_jb_j=\sum_{j=1}^na_j(B_j-B_{j-1})
=\sum_{j=1}^{n-1}B_j(a_j-a_{j+1}).
\]
Using monotonicity of $a_1,\ldots,a_n$, we get
\[
|S|\le\sum_{j=1}^{n-1}M|a_j-a_{j+1}|=M|a_1-a_n|.
\]
\end{proof}

The next theorem shows that a constant
$K<1.78$ can be obtained in~\eqref{eq:thebound}.
\begin{thm}\label{thm:partial}
Let $n\in\Nats$ and let $\lambda\in\Vc_n$.
Then there is a permutation $\sigma\in S_n$ such that
\begin{equation}\label{eq:main}
\|T_{\lambda\circ\sigma}\|\le K\|\lambda\|_\infty
\end{equation}
with $K=\frac12+\frac4\pi$.
\end{thm}
\begin{proof}
Since $\lambda$ is the eigenvalue sequence of $T_{\lambda\circ\sigma}+T_{\lambda\circ\sigma}^*$ for every $\sigma$,
we have $\|\RealPart T_{\lambda\circ\sigma}\|=\frac12\|\lambda\|_\infty$.
Thus, it will suffice to find $\sigma$ so that
\[
\|\i(T_{\lambda\circ\sigma}-T_{\lambda\circ\sigma}^*)\|\le\frac8\pi\|\lambda\|_\infty.
\]
Using~\eqref{eq:td}, we have
\begin{align*}
\i(T_\lambda-T_\lambda^*)&=\i
\begin{pmatrix}
0 &t_1&t_2&\cdots& t_{n-1} \\
-\overline{t_1}&0&t_1& \ddots & \vdots \\
-\overline{t_2}&-\overline{t_1}&0&\ddots& t_2 \\
\vdots& \ddots &\ddots&\ddots& t_1 \\[1ex]
-\overline{t_{n-1}} &\cdots & -\overline{t_2} & -\overline{t_1} & 0 \hspace*{0.2em}
\end{pmatrix} \displaybreak[2] \\[2ex]
&=\frac{\i}n\sum_{p=1}^n\lambda_p
\begin{pmatrix}
0 &\omega_n^{p-1}&\omega_n^{2(p-1)}&\cdots&  \omega_n^{(n-1)(p-1)} \\
-\overline{\omega_n}^{p-1}&0&\omega_n^{p-1}& \ddots &  \vdots \\
-\overline{\omega_n}^{2(p-1)}&-\overline{\omega_n}^{p-1}&0&\ddots&  \omega_n^{2(p-1)} \\
\vdots&\ddots&\ddots&\ddots&  \omega_n^{p-1} \\[1ex]
-\overline{\omega_n}^{(n-1)(p-1)} &\cdots & -\overline{\omega_n}^{2(p-1)} & -\overline{\omega_n}^{p-1} &  0
\end{pmatrix} \\[2ex]
&=\sum_{p=1}^{n}\lambda_p(W_n^*)^{p-1}H_nW_n^{p-1}.
\end{align*}
Using now Remark~\ref{rem:Hn}, we have
\[
\i(T_\lambda-T_\lambda^*)
=\sum_{p=1}^n\lambda_p\sum_{k=1}^n\mu_kW_n^{k-p+1}Q_nW_n^{p-k-1}
=\sum_{l=1}^n\left(\sum_{k=1}^n\lambda_{k-l+1}\mu_k\right)W_n^l Q_n(W_n^*)^l,
\]
where $k-l+1$ in the subscript of $\lambda$ is taken modulo $n$ in the range from $1$ to $n$.
Consequently,
\[
\|T_\lambda-T_\lambda^*\|=\max_{1\le l\le n}\left|\sum_{k=1}^n\lambda_{k-l+1}\mu_k\right|
=\max_{1\le l\le n}\left|(\lambda\circ\rho_n^{l-1})\cdot\mu\right|,
\]
where $\cdot$ represents the usual scalar product, $\rho_n$ is the full cycle permutation $\rho_n(j)=j-1$ (mod $n$)
and $\mu=(\mu_1,\ldots,\mu_n)$.
We seek a permutation $\sigma$ making
\[
\max_{1\le l\le n}|(\lambda\circ\sigma\circ\rho_n^{l-1})\cdot\mu|
\le\frac8\pi\|\lambda\|_\infty,
\]
since the quantity on the left is $\|T_{\lambda\circ\sigma}-T_{\lambda\circ\sigma}^*\|$.
Since we run through all rotations $\rho_n^{l-1}$, we may without loss of generality replace $\mu$ by $\mu\circ\rho_n^m$
for any $m$.
From Remark~\ref{rem:muk}, we see that some such $\mu\circ\rho_n^m$ is monotone, with largest element $<2/\pi$
and smallest element $>-2/\pi$.
Now we choose $\sigma$ so that all partial sums of $\lambda\circ\sigma$ are of absolute value $\le\|\lambda\|_\infty$.
This implies that all partial sums of all rotations $\lambda\circ\sigma\circ\rho_n^{l-1}$ are of absolute value $\le2\|\lambda\|_\infty$.
Now Lemma~\ref{lem:DIrichlet} implies
\[
\|T_{\lambda\circ\sigma}-T_{\lambda\circ\sigma}^*\|
\le \frac8\pi\|\lambda\|_\infty.
\]
\end{proof}

The following result demonstrates that some rearrangement is required to get a bounded constant $K$ in \eqref{eq:thebound}.
Although this sort of calculation (to get a lower bound for the norm of the upper triangular projection)
was also made in~\cite{ACN92} for upper triangular Toeplitz matrices, these were not of the form $T_\lambda$
for $\lambda\in\Vc_n$
(see Remark~\ref{rem:Tlambdas}).

\begin{prop}\label{prop:lowerbd}
For $\lambda=(\underset{n}{\underbrace{1,\cdots,1}},\underset{n}{\underbrace{-1,\cdots,-1}})$,
we have $\|T_\lambda\|\ge\frac1\pi\log(n)-\frac3{2\pi}$.
\end{prop}
\begin{proof}
For the given $\lambda$,
\begin{align*}
t_d&=\frac{1}{n}\sum_{k=1}^n\omega_{2n}^{d(k-1)}-\frac{1}{n}\sum_{k=n+1}^{2n}\omega_{2n}^{d(k-1)}
=\frac1n\cdot\frac{1-\omega_{2n}^{nd}}{1-\omega_{2n}^d}\\
&=\frac1n\cdot\frac{1-(-1)^d}{1-\omega_{2n}^d}.
\end{align*}
Hence,
\begin{align}
\label{eq:tdgrowth}
t_d=\begin{cases}
0 & \text{ if } d \text{ is even}\\
\frac1n\left(1+\i\cot\left(\frac{\pi d}{2n}\right)\right) & \text{ if } d \text{ is odd}.
\end{cases}
\end{align}
We will estimate from below the quadratic form of $T_\lambda$ on the vector $g=\frac{1}{\sqrt{2n}}(\underset{2n}{\underbrace{1,1,\dots,1}})$.
\begin{align}
\label{eq:quadform}
\left<T_\lambda g,g\right>
=\left<\left(\sum_{k=1}^{2n-1}t_k,\sum_{k=1}^{2n-2}t_k,\dots,t_1+t_2,t_1,0\right),g\right>
=\frac{1}{2n}\sum_{k=1}^{2n-1}(2n-k)t_k.
\end{align}
From \eqref{eq:tdgrowth} and \eqref{eq:quadform}, we obtain
\begin{align*}
\re\left<T_\lambda g,g\right>=\frac{1}{2n^2}\sum_{\substack{1\le j\leq 2n-1 \\ j\text{ odd}}}(2n-j)
=\frac{1}{2},
\end{align*}
\begin{align*}
\im\left<T_\lambda g,g\right>&=\frac{1}{2n^2}\sum_{\substack{1\le j\leq 2n-1 \\ j\text{ odd}}}(2n-j)
\cot\left(\frac{\pi j}{2n}\right) \displaybreak[2]\\
&=\frac{1}{2n^2}\bigg(\sum_{\substack{1\le j\leq n-1 \\ j\text{ odd}}}+\sum_{\substack{n+1\le j\leq 2n-1 \\ j\text{ odd}}}\bigg)(2n-j)\cot\left(\frac{\pi j}{2n}\right) \displaybreak[2]\\
&=\frac{1}{2n^2}\sum_{\substack{1\le j\leq n-1 \\ j\text{ odd}}}(2n-2j)\cot\left(\frac{\pi j}{2n}\right).
\end{align*}
Application of Lemma \ref{lemma2} implies
\begin{align*}
\im\left<T_\lambda g,g\right>&
=\frac{1}{\pi n}\sum_{\substack{1\le j\leq n-1 \\ j\text{ odd}}}\frac{2n-2j}{j}
-\frac{1}{2 n^2}\sum_{\substack{1\le j\leq n-1 \\ j\text{ odd}}}(2n-2j)f_1\left(\frac{j}{2n}\right) \\ 
&\ge\frac{1}{\pi n}\sum_{\substack{1\le j\leq n-1 \\ j\text{ odd}}}\frac{2n-2j}{j}
-\frac{1}{\pi n^2}\sum_{\substack{1\le j\leq n-1 \\ j\text{ odd}}}(2n-2j) \\
&=\frac2\pi\sum_{\substack{1\le j\leq n-1 \\ j\text{ odd}}}\frac1j
-\frac4{\pi n}\sum_{\substack{1\le j\leq n-1 \\ j\text{ odd}}}1
+\frac{2}{\pi n^2}\sum_{\substack{1\le j\leq n-1 \\ j\text{ odd}}}j
\end{align*}
and the standard computations
\[
\sum_{\substack{1\le j\le r\\j\text{ odd}}}\frac1j>\frac12\log(r+1),
\qquad
\sum_{\substack{1\le j\le r\\ j\text{ odd}}}j=\left\lfloor\frac{r+1}2\right\rfloor^2
\]
(both for arbitrary $r\in\Nats$, where $\lfloor\cdot\rfloor$ is the floor function)
yield
\[
\im\left<T_\lambda g,g\right>
>\frac1\pi\log(n)-\frac3{2\pi}.
\]
\end{proof}

We will now consider the conjugation with Fourier matrices with a view to taking inductive limits
of matrix algebras.
Let $(e_{ij}^{(n)})_{1\le i,j\le n}$ be the standard system of matrix units for $M_n$.
Let $\Theta_n:M_n\to\UT_n^{(1)}$ be the projection given by
\[
\Theta_n(e_{ij}^{(n)})=\begin{cases}e_{ij}^{(n)}&\text{if }i<j, \\0&\text{otherwise.}\end{cases}
\]
Let $\alpha_n:M_n\to M_n$ be the inner automorphism $\alpha_n(A)=U_n^*AU_n$,
where $U_n$ is the Fourier matrix as described above equation~\eqref{eq:number}.
Thus, for $\lambda=(\lambda_1,\ldots,\lambda_n)\in D_n$ with $\sum_{j=1}^n\lambda_j=0$,
we have $T_\lambda=\Theta_n(\alpha_n(\lambda))\in\UT_n^{(1)}$.

Let $m,n\in\Nats$ and consider the inclusion
$\gamma_{m,n}:M_m\to M_{mn}$ given by
\[
\gamma=\gamma_{m,n}:e_{ij}^{(m)}\mapsto\sum_{k=1}^ne_{n(i-1)+k,n(j-1)+k}^{(mn)}.
\]
Under the common identification of $M_{mn}$ with $M_m\otimes M_n$,  $\gamma_{m,n}(x)$ is identified with $x\otimes I_n$.

Let $\beta=\beta_{m,n}=\alpha_{mn}^{-1}\circ\gamma\circ\alpha_{m}:M_m\to M_{mn}$.
While $\beta(e_{ij}^{(m)})$ is not very pretty to describe for general $i$ and $j$,
as seen in the following lemma, the restriction of $\beta$ to the diagonal subalgebra $D_m$ is rather nice;
it is the flip of the usual tensor product embedding.

\begin{lemma}\label{lem:commdiag}
The map $\gamma_{m,n}$ sends Toeplitz matrices into Toeplitz matrices.
For each $p\in\{0,1,\ldots,m\}$ we have
\begin{equation}\label{eq:gammaUT}
\gamma_{m,n}(\UT_m^{(p)})\subseteq\UT_{mn}^{(pn)},\qquad\gamma_{m,n}(\UTTM_m^{(p)})\subseteq\UTTM_{mn}^{(pn)}.
\end{equation}
Thus, we have the commuting diagram
\begin{equation}\label{eq:gambeta}
\xymatrix{
\UT_m^{(1)} \ar[r]^{\gamma} & \UT_{mn}^{(1)} \\
M_m \ar[r]^\gamma \ar[u]^{\Theta_m} & M_{mn} \ar[u]_{\Theta_{mn}} \\
M_m \ar[r]^\beta \ar[u]^{\alpha_m} & M_{mn} \ar[u]_{\alpha_{mn}} \\
\rule{0ex}{2.5ex}D_m \ar[r]^{\beta} \ar@{^{(}->}[u] & \rule{0ex}{2.5ex}D_{mn} \ar@{^{(}->}[u]
}
\end{equation}
where the top and bottom row arrows indicate the restriction of $\gamma$ to $\UT_m^{(1)}$ and, respectively, $\beta$ to $D_m$.
Finally, for $e_{ii}^{(m)}\in D_m$, we have
\begin{equation}\label{eq:betaeii}
\beta(e_{ll}^{(m)})=\sum_{k=1}^ne_{m(k-1)+l,m(k-1)+l}^{(mn)}\,.
\end{equation}
\end{lemma}
\begin{proof}
The inclusions~\eqref{eq:gammaUT} are easy to verify, and
we need only show~\eqref{eq:betaeii}, which we do by verifying
\begin{equation}\label{eq:gamalph}
\gamma\circ\alpha_m(e_{ll}^{(m)})=\alpha_{mn}(\sum_{k=1}^ne_{m(k-1)+l,m(k-1)+l}^{(mn)}).
\end{equation}
Using~\eqref{eq:Dff}, we have
\begin{equation}\label{eq:ijksum}
e_{ll}^{(m)}\overset{\alpha_m}\mapsto\frac1m\sum_{1\le i,j\le m}\omega_m^{(l-1)(j-i)}e_{ij}^{(m)}
\overset\gamma\mapsto\frac1m\sum_{\substack{1\le i,j\le m\\ 1\le k\le n}}\omega_m^{(l-1)(j-i)}
 e_{n(i-1)+k,n(j-1)+k}^{(mn)}\,,
\end{equation}
while
\begin{multline}\label{eq:absum}
\alpha_{mn}(\sum_{k=1}^ne_{m(k-1)+l,m(k-1)+l}^{(mn)})=
\frac1{mn}\sum_{\substack{1\le k\le n\\ 1\le a,b\le mn}}\omega_{mn}^{(m(k-1)+l-1)(b-a)}e_{ab}^{(mn)} \\
=\frac1m\sum_{1\le a,b\le mn}\omega_{mn}^{(l-1)(b-a)}e_{ab}^{(mn)}\left(\frac1n\sum_{1\le k\le n}\omega_n^{(k-1)(b-a)}\right).
\end{multline}
But
\[
\frac1n\sum_{1\le k\le n}\omega_n^{(k-1)(b-a)}=
\begin{cases}1,&\text{if }n\text{ divides }b-a \\ 0,&\text{if }n\text{ does not divide }b-a,\end{cases}
\]
and $n$ divides $b-a$ if and only if we have $a=n(i-1)+k$ and $b=n(j-1)+k$ for some $1\le i,j\le m$ and some $1\le k\le n$,
so the nested summation~\eqref{eq:absum} equals the right--most summation in~\eqref{eq:ijksum},
and~\eqref{eq:gamalph} is verified.
\end{proof}

For applications in the setting of inductive limits of maps as in the diagram~\eqref{eq:gambeta},
we will want
a version of Theorem~\ref{thm:partial} but for elements of $D_{mn}$ that are orthogonal to $\beta(D_m)$ and taking only
reorderings of diagonal entries that fix $\beta(D_m)$.
This is provided by the next result in the case $n=2$.

\begin{prop}\label{prop:Dm2}
Fix $m\in\{2,3,\ldots\}$ and let
$\lambda^{(i)}_j\in\Reals$ for $1\le i\le m$ and $1\le j\le 2$ satisfy $\lambda^{(i)}_1+\lambda^{(i)}_2=0$ for all $i$.
Thus,
\[
\lambda:=(\lambda^{(1)}_1,\lambda^{(2)}_1,\ldots,\lambda^{(m)}_1,
\lambda^{(1)}_2,\lambda^{(2)}_2,\ldots,\lambda^{(m)}_2)\in D_{2m}\ominus\beta_{m,2}(D_m).
\]
Then there are permutations $\sigma_1,\ldots,\sigma_m$ of $\{1,2\}$
such that, considering the reordering
\[
\kappa=
(\lambda^{(1)}_{\sigma_1(1)},\lambda^{(2)}_{\sigma_2(1)},\ldots,\lambda^{(m)}_{\sigma_m(1)},
\lambda^{(1)}_{\sigma_1(2)},\lambda^{(2)}_{\sigma_2(2)},\ldots,\lambda^{(m)}_{\sigma_m(2)}),
\]
of $\lambda$, we have
$\|T_\kappa\|\le C\|\lambda\|_\infty$,
where $C=\frac12+\frac{12}\pi$.
\end{prop}
\begin{proof}
Proceeding as in the proof of Theorem~\ref{thm:partial} we have
\[
\|T_\kappa\|\le \frac12\|\lambda\|_\infty+\frac4\pi R,
\]
where
\begin{equation}\label{eq:maxsum}
R=\max_{\substack{1\le k\le2m \\1\le p\le2m}}\left|\sum_{j=1}^k\kappa\circ\rho_{2m}^p(j)\right|
\end{equation}
is the maximum absolute value of all partial sums of all rotations of $\kappa$,
i.e., where $\rho_{2m}$ is the full cycle permutation $k\mapsto k-1$ (modulo $2m$) of $\{1,\ldots,2m\}$.
Since $\lambda^{(i)}_1=-\lambda^{(i)}_2$ for all $i$, we may choose $\sigma_1$ arbitrarily and then choose $\sigma_2,\ldots,\sigma_m$
recursively so that the sign of $\lambda^{(k)}_{\sigma_k(1)}$ is the opposite of the sign of
$\sum_{i=1}^{k-1}\lambda^{(i)}_{\sigma_i(1)}$.
This ensures
\[
\left|\sum_{i=1}^k\lambda^{(i)}_{\sigma_i(1)}\right|\le\|\lambda\|_\infty
\]
for all $k\in\{1\ldots,m\}$.
This, in turn, implies
\[
\left|\sum_{i=k}^{l}\lambda^{(i)}_{\sigma_i(j)}\right|\le2\|\lambda\|_\infty
\]
for all $1\le k\le l\le m$ and all $j\in\{1,2\}$.
Together, these estimates yield $R\le3\|\lambda\|_\infty$.
\end{proof}

\section{Applications using finite type I von Neumann algebras}
\label{sec:vN}

This section is concerned with applications of Theorem~\ref{thm:partial}
to constructing quasinilpotent and related elements
in finite type I von Neumann algebras, and also constructions in II$_1$ factors that result from this.
Throughout, $K$ will denote the constant from Theorem~\ref{thm:partial}.

\begin{prop}\label{obs:prod}
Let
$\Mcal_j\subseteq\Bc(\HEu_j)$ be a von Neumann algebra ($j\in J$).
Let $\Mcal=\prod_{j\in J}\Mcal_j$ be the direct product of von Neumann algebras, so that
$\Mcal\subseteq\Bc(\bigoplus_{j\in J}\HEu_j)$ in the canonical way.
Let $x=(x_j)_{j\in J}\in\Mcal$.
Suppose each  $x_j$ is quasinilpotent.
\begin{enumerate}[(i)]
\item Then $x$ is s.o.t.--quasinilpotent.
\item The element $x$ is quasinilpotent if and only if
\begin{equation}\label{eq:xqn}
\lim_{n\to\infty}\left(\sup_{j\in J}\|x_j^n\|\right)^{1/n}=0.
\end{equation}
\item If each $x_j$ is nilpotent with $x_j^{n(j)}=0$ for $n(j)\in\Nats$
and if
\[
\lim_{N\to\infty}\left(\sup\{\|x_j\|\mid j\in J,\, n(j)>N\}\right)=0,
\]
then $x$ is quasinilpotent.
\end{enumerate}
\end{prop}
\begin{proof}
If $\xi=(\xi_j)_{j\in J}\in\bigoplus_{j\in J}\HEu_j$ with $\xi_j=0$ for all $j\in J\backslash F$, where $F$ is a finite
subset of $J$, then
\[
\|((x^*)^nx^n)^{1/2n}\xi\|\le\left(\max_{j\in F}\|x_j^n\|^{1/n}\right)\|\xi\|\to0\text{ as }n\to\infty.
\]
This implies
\[
\text{s.o.t.--}\lim_{n\to\infty}((x^*)^nx^n)^{1/2n}=0,
\]
proving~(i).

Assertion~(ii) results from the formula for $\|x^n\|$,
while the hypothesis of~(iii) implies~\eqref{eq:xqn}.
\end{proof}

\begin{lemma}\label{lem:typeIn}
Let $n\in\Nats$ and let $\Mcal=L^\infty(X,\nu)\otimes M_n(\Cpx)$ be a type I$_n$ von Neumann algebra with separable predual
and let $a=a^*\in\Mcal$.
Then $a$ is the real part of a quasinilpotent element in $\Mcal$ if and only if the center--valued trace of $a$ is zero.
In this case, there is $z\in\Mcal$ with $z^*+z=a$, $z^n=0$ and $\|z\|\le K\|a\|$.
\end{lemma}
\begin{proof}
We identify $\Mcal$ with the bounded, $\nu$--measurable functions $X\to M_n(\Cpx)$.
Then the center--valued trace of $a\in\Mcal$ is just the scalar valued function $\tr_n(a(x))$.
If $a=z+z^*$ for $z\in\Mcal$ with $z$ quasinilpotent, then for almost every $x\in X$,
$z(x)$ will be nilpotent and its $n$th power must vanish.
In particular, the matrix trace of $z(x)$ will vanish for almost every $x$;
consequently, the center--valued trace of $a$ is zero.

If $a=a^*\in\Mcal$, then
by standard arguments
we can choose a $\nu$--measurable
unitary--valued function $V_n:X\to M_n(\Cpx)$ so that $V_n(x)a(x)V_n(x)^*$ is diagonal for all $x\in X$.
Let $\lambda(x)$ be the diagonal entries, i.e., $V_n(x)a(x)V_n(x)^*=\diag(\lambda(x))$.
If the center--valued trace of $a$ is zero, then the sum of $\lambda(x)$ is zero
(for almost every $x$) and since we may change $V_n(x)$  in a measurable way to re--order the diagonal elements
as needed,
using Theorem~\ref{thm:partial},
we have $\|T_{\lambda(x)}\|\le K\|a(x)\|$.
Then $z(x)=V_n(x)^*T_{\lambda(x)}V_n(x)$ is the desired nilpotent element.
\end{proof}

Combining Lemma~\ref{lem:typeIn} and Proposition~\ref{obs:prod}, we obtain the following,
which is a partial answer to Question~\ref{ques:typeI}.

\begin{prop}\label{prop:typeI}
Let $\Mcal$ be a finite type~I von Neumann algebra with separable predual.
We may write
\[
\Mcal=\prod_{n\in J}L^\infty(X_n,\nu_n)\otimes M_n(\Cpx)
\]
for some $J\subseteq\Nats$ and some nonzero finite measure $\nu_n$.
If $a=z+z^*$ for $z\in\Mcal$ an s.o.t.--quasinilpotent element, then the center--valued trace of $a$ is zero.

Conversely, suppose $a=a^*\in\Mcal$
and that the center--valued trace of $a$ is zero.
\begin{enumerate}[(i)]
\item Then $a=z+z^*$ for an s.o.t.--quasinilpotent element $z\in\Mcal$ with $\|z\|\le K\|a\|$.
\item
If $J$ is finite or if $J$ is infinite but $a=a^*=(a_n)_{n\in J}\in\Mcal$ with
\[
a_n\in L^\infty(X_n,\nu_n)\otimes M_n(\Cpx)\quad\text{and}\quad
\lim_{J\ni n\to\infty}\|a_n\|=0,
\]
then
there is quasinilpotent element $z\in\Mcal$ with $z+z^*=a$ and $\|z\|\le K\|a\|$.
\end{enumerate}
\end{prop}

It looks like further progress in answering Question~\ref{ques:typeI} using these
techniques involving upper triangular Toeplitz matrices
could be made only with better understanding of the behavior
of $\|T_{\lambda}^n\|^{1/n}$ for large $n$ and long $\lambda$.

\medskip
By embedding finite type I von Neumann algebras into II$_1$--factors and using Proposition~\ref{prop:typeI},
one can obtain many examples
of self--adjoint elements in II$_1$--factors that are real parts of quasinilpotents.
Recall that the distribution of a self--adjoint element of a II$_1$--factor is
the probability measure that is the trace composed with spectral measure.
\begin{prop}\label{prop:R}
Let $R$ be the hyperfinite II$_1$--factor
and let $D\subset R$ be its Cartan (i.e., diagonal) maximal abelian self--adjoint subalgebra.
Suppose a compactly supported Borel probability measure $\nu$ on $\Reals$ is of the form
\begin{equation}\label{eq:nu}
\nu=\sum_{n\in J}\frac1{n}\int(\delta_{f_{n,1}(t)}+\delta_{f_{n,2}(t)}+\cdots+\delta_{f_{n,n}(t)})\,d\nu_n(t),
\end{equation}
where $J\subseteq\Nats$ or $J=\Nats$,
where each $\nu_n$ is a nonzero positive measure on a standard Borel space $X$ with $\sum_{j\in J}\nu_j(X)=1$
and where $f_{n,1},\ldots,f_{n,n}$ are real--valued measurable functions on $X$ such that for $\nu_n$--almost every $x$ we have
$f_{n,1}(x)+\cdots+f_{n,n}(x)=0$.

\begin{enumerate}[(i)]
\item
Then there is an s.o.t.--quasinilpotent element $z\in R$ such that $a:=z+z^*\in D$, $\|z\|\le K\|a\|$ and the distribution
$\mu_a$ is equal to $\nu$.
\item
If $J$ is finite, then the element $z$ can be chosen to be nilpotent.
\item
Suppose $J$ is infinite and let
\[
M_n=max\{\|f_{n,1}\|_\infty,\ldots,\|f_{n,n}\|_\infty\},
\]
where the norms are in $L^\infty(\nu_n)$.
If $\lim_{j\to\infty}M_j=0$,
then the element $z$ can be chosen to be quasinilpotent.
\end{enumerate}
\end{prop}
\begin{proof}
We can realize $D$ as a copy of
\[
\prod_{n\in J}L^\infty(\nu_n)^{\oplus n}
\]
in $R$, and using partial isometries from $R$ we can find a type I subalgebra $\Mcal$
with $D\subseteq\Mcal\subset R$ of the form
\[
\Mcal=\prod_{n\in J}L^\infty(\nu_n)\otimes M_{n}(\Cpx),
\]
where identifying each $L^\infty(\nu_n)\otimes M_{n}(\Cpx)$ with the $M_n(\Cpx)$--valued $\nu_n$--measurable functions,
$D$ is identified with the product of the sets of functions taking values in the diagonal matrices.
The element $a=\big(\diag(f_{n,1}(\cdot),\ldots,f_{n,n}(\cdot))\big)_{n\in J}$ belongs to $D$,
has center--valued trace in $\Mcal$ equal to zero and has distribution $\nu$.
Now we apply Proposition~\ref{prop:typeI}, to find $z\in\Mcal$ having the desired properties.
\end{proof}

The question of whether every self--adjoint element of $a\in D$ having distribution $\nu$ as in the above proposition is the
real part of a quasinilpotent remains unanswered in general, though it is not hard to show that
if the essential ranges of the functions $(f_{n,i})_{n\in J,\,1\le i\le n}$
are pairwise disjoint, then the answer is yes, by the construction used above.
The similar question for arbitrary self--adjoint elements of $R$ is even less clear.
However, in the ultrapower of the hyperfinite II$_1$--factor,
the answer is yes.

\begin{prop}
Let $R^\omega$ be an ultrapower of the hyperfinite II$_1$ factor, for $\omega$ a non--principle ultrafilter on $\Nats$.
Let $a=a^*\in R^\omega$ and suppose the distribution of $a$ is of the form $\nu$ as in~\eqref{eq:nu}.
\begin{enumerate}[(i)]
\item
Then there is an s.o.t.--quasinilpotent element $z\in R^\omega$ such that $a=z+z^*$ and $\|z\|\le K\|a\|$.
\item
If $J$ is finite, then the element $z$ can be chosen to be nilpotent.
\item
Suppose $J$ is infinite and let
\[
M_n=max\{\|f_{n,1}\|_\infty,\ldots,\|f_{n,n}\|_\infty\},
\]
where the norms are in $L^\infty(\nu_n)$.
If $\lim_{j\to\infty}M_j=0$,
then the element $z$ can be chosen to be quasinilpotent.
\end{enumerate}
\end{prop}
\begin{proof}
Since $R\subseteq R^\omega$ as a unital W$^*$--subalgebra, using Proposition~\ref{prop:R}, there is $b=b^*\in R^\omega$
whose distribution is $\nu$ and with s.o.t.--quasinilpotent $y\in R^\omega$ such that $b=y+y^*$ and $\|y\|\le K\|b\|$,
and according with the additional stipulations of~(ii) and~(iii) in the case that the corresponding hypotheses are satisfied.
Since all self--adjoint elements in $R^\omega$ having given distribution are unitarily equivalent, we find $z$ as a unitary conjugate of $y$.
\end{proof}

For a purely spectral condition that is sufficient for a self--adjoint to be  the real part of a quasinilpotent, valid in all II$_1$--factors,
we turn to discrete measures.
\begin{prop}
Let $\Mcal$ be a II$_1$--factor with trace $\tau$ and let $a=a^*\in\Mcal$ with $\tau(a)=0$.
Suppose that the distribution $\mu_a$ of $a$ is a discrete measure that can be written
\[
\mu_a=\sum_{i\in I}s_i\left(\frac1{n(i)}\sum_{k=1}^{n(i)}\delta_{t(i,k)}\right),
\]
where for all $i\in I$, $s_i>0$, $n(i)\in\Nats$, $t(i,k)\in\Reals$, $\sum_{k=1}^{n(i)}t(i,k)=0$
and where $\delta_t$ denotes the Dirac measure at $t$ and  $\sum_{i\in I}s_i=1$.
\begin{enumerate}[(i)]
\item
Then there is an s.o.t.--quasinilpotent element $z\in\Mcal$ such that $a=z+z^*$ and $\|z\|\le K\|a\|$.
\item
If $\sup_{i\in I}n(i)<\infty$, then the element $z$ can be chosen to be nilpotent.
\item
Let
$M_i=max\{|t(i,1)|,\ldots,|t(i,n(i))|\}$.
If
\[
\lim_{N\to\infty}(\sup\{M_i\mid i\in I,\,n(i)>N\})=0,
\]
then the element $z$ can be chosen to be quasinilpotent.
\end{enumerate}
\end{prop}
\begin{proof}
By Proposition~\ref{prop:R}, there is a quasinilpotent element $y\in R$ such that the distribution of $y+y^*$ equals $\mu_a$
and $\|y\|\le K\|a\|$.
There is a copy of $R$ embedded as a unital W$^*$--subalgebra of $\Mcal$.
Since the spectral measure of $a$ is discrete, all self--adjoint elements in $\Mcal$
having this spectral measure are unitarily equivalent in $\Mcal$.
Thus, a unitary conjugate of $y$ is the desired element $z$.
\end{proof}

\section{Applications using inductive limits}
\label{sec:indlim}

In this section we will apply Proposition~\ref{prop:Dm2} in the setting of inductive limits
of maps like the ones in~\eqref{eq:gambeta} of Lemma~\ref{lem:commdiag}, to
conclude that some self--adjoint elements of the Cartan masa in the hyperfinite II$_1$--factor $R$
whose distributions are of a certain form, are the real parts of quasinilpotent elements in $R$.

We will use the following easy result to construct quasinilpotent elements.
\begin{prop}\label{prop:zsum}
Let $z_1,z_2,\ldots$ be pairwise commuting quasinilpotent elements in a Banach algebra $B$
and suppose $\sum_{j=1}^\infty\|z_j\|<\infty$.
Let $z=\sum_{j=1}^\infty z_j$.
Then $z$ is quasinilpotent.
\end{prop}
\begin{proof}
Without loss of generality we may take $B$ unital.
Let $A$ be the unital Banach subalgebra generated by
$z_1,z_2,\cdots$. Then $A$ is an abelian algebra. We need only to
show that the spectrum of $z$ relative to $A$ is $\{0\}$ since it
is equivalent to $\lim_{n\rightarrow\infty}\|z^n\|^{1/n}=0$.
Using 
the Gelfand transform,
\[
\sigma_A(z)=\{\varphi(z):\,\varphi\text{ is a multiplicative linear functional of }A\}.
\]
Since $z_n$ is quasinilpotent, we have
$\varphi(z_n)=0$ for every multiplicative linear functional $\varphi$ on $A$.
Since multiplicative linear functionals are automatically bounded, we have
\[
\varphi(z)=\lim_{n\rightarrow\infty}\sum_{k=1}^n\varphi(z_k)=0,
\]
which proves the lemma.
\end{proof}

Consider the Cartan masa (maximal abelian self--adjoint subalgebra) $D$ of the hyperfinite II$_1$--factor $R$,
realized as the inductive limit of the trace--preserving maps shown below,
\begin{equation}\label{eq:DinRindlim}
\xymatrix{
M_{n_1}\ar[r]^{\beta^{(1)}} & M_{n_1n_2}\ar[r]^{\beta^{(2)}} & M_{n_1n_2n_3}\ar[r]^{\beta^{(3)}\quad} &
  \cdots M_{n_1n_2\cdots n_j} \ar[r]^{\quad\;\beta^{(j)}} & \cdots R \\
\rule{0ex}{2.5ex}D_{n_1}\ar[r]^{\beta^{(1)}}\ar@{^{(}->}[u] & \rule{0ex}{2.5ex}D_{n_1n_2}\ar[r]^{\beta^{(2)}}\ar@{^{(}->}[u] &
 \rule{0ex}{2.5ex}D_{n_1n_2n_3}\ar[r]^{\beta^{(3)}\quad}\ar@{^{(}->}[u] &
 \cdots\rule{0em}{2.5ex}D_{n_1n_2\cdots n_j}\ar[r]^{\quad\;\beta^{(j)}}\ar@<-2ex>@{^{(}->}[u] &
 \cdots\rule{0ex}{2.5ex}D,\ar@<-2ex>@{^{(}->}[u]
}
\end{equation}
where $n_1,n_2,\ldots\in\{2,3,\ldots\}$ and $\beta^{(j)}$ is the map $\beta_{n_1n_2\cdots n_j,n_1n_2\cdots n_jn_{j+1}}$ defined above
Lemma~\ref{lem:commdiag}, and whose restriction to the diagonal subalgebra $D_{n_1n_2\cdots n_j}$ is as described in
Lemma~\ref{lem:commdiag}.

\begin{lemma}\label{lem:Tasum}
Suppose $a_j=a_j^*\in D_{n_1\cdots n_j}$ are such that $\tau(a_j)=0$ for all $j$.
Let $T_{a_j}=\Theta_{n_1\cdots n_j}\circ\alpha^{(j)}(a_j)\in\UTTM_{n_1\cdots n_j}^{(1)}$ and suppose $\sum_{j=1}^\infty\|T_{a_j}\|<\infty$.
Then the series $a:=\sum_{j=1}^\infty a_j\in D$ converges in norm and
there is a quasinilpotent operator $z\in R$ such that
$z^*+z=a$ and $\|z\|\le\sum_{j=1}^\infty\|T_{a_j}\|$.
\end{lemma}
\begin{proof}
Using Lemma~\ref{lem:commdiag}, we have the big commuting diagram
\begin{equation}\label{eq:bigdiag}
\xymatrix{
\UT_{n_1}\ar[r]^{\gamma^{(1)}}
 & \UT_{n_1n_2}\ar[r]^{\gamma^{(2)}}
 & \UT_{n_1n_2n_3}\ar[r]^{\gamma^{(3)}\quad}
 & \cdots \UT_{n_1n_2\cdots n_j} \ar[r]^{\quad\;\gamma^{(j)}}
 & \cdots\subset R \\
M_{n_1}\ar[r]^{\gamma^{(1)}}\ar[u]_{\Theta^{(1)}}
 & M_{n_1n_2}\ar[r]^{\gamma^{(2)}}\ar[u]_{\Theta^{(2)}}
 & M_{n_1n_2n_3}\ar[r]^{\gamma^{(3)}\quad}\ar[u]_{\Theta^{(3)}}
 & \cdots M_{n_1n_2\cdots n_j} \ar[r]^{\quad\;\gamma^{(j)}}\ar@<-2ex>[u]_{\Theta^{(j)}}
 & \cdots R \\
M_{n_1}\ar[r]^{\beta^{(1)}}\ar[u]_{\alpha^{(1)}}
 & M_{n_1n_2}\ar[r]^{\beta^{(2)}}\ar[u]_{\alpha^{(2)}}
 & M_{n_1n_2n_3}\ar[r]^{\beta^{(3)}\quad}\ar[u]_{\alpha^{(3)}}
 & \cdots M_{n_1n_2\cdots n_j} \ar[r]^{\quad\;\beta^{(j)}}\ar@<-2ex>[u]_{\alpha^{(j)}}
 & \cdots R\ar@<-2ex>[u]_{\alpha} \\
\rule{0ex}{2.5ex}D_{n_1}\ar[r]^{\beta^{(1)}}\ar@{^{(}->}[u]
 & \rule{0ex}{2.5ex}D_{n_1n_2}\ar[r]^{\beta^{(2)}}\ar@{^{(}->}[u]
 & \rule{0ex}{2.5ex}D_{n_1n_2n_3}\ar[r]^{\beta^{(3)}\quad}\ar@{^{(}->}[u]
 & \cdots\rule{0em}{2.5ex}D_{n_1n_2\cdots n_j}\ar[r]^{\quad\;\beta^{(j)}}\ar@<-2ex>@{^{(}->}[u]
 & \cdots\rule{0ex}{2.5ex}D,\ar@<-2ex>@{^{(}->}[u]
}
\end{equation}
where $\gamma^{(j)}=\gamma_{n_1n_2\cdots n_j,n_1n_2\cdots n_jn_{j+1}}$ is the usual inclusion of tensor products,
where $\alpha^{(j)}=\alpha_{n_1n_2\cdots n_j}$ is the automorphism implemented by conjugation with the Fourier
matrix and their inductive limit $\alpha$ is the resulting isomorphism between copies of the hyperfinite II$_1$--factor,
and where $\Theta^{(j)}=\Theta_{n_1n_2\cdots n_j}$ is the upper triangular projection.

Since upper triangular Toeplitz matrices commute with each other,
and taking into account the observation~\eqref{eq:gammaUT} of Lemma~\ref{lem:commdiag},
by Proposition~\ref{prop:zsum} the series $\zt:=\sum_{j=1}^\infty T_{a_j}$ converges in norm to a quasinilpotent operator in $R$.
By construction, we have $T_{a_j}+T_{a_j}^*=\alpha^{(j)}(a_j)$, so the series $a=\sum_{j=1}^\infty a_j$ converges in norm,
and $a=z^*+z$, where $z=\alpha^{-1}(\zt)$.
\end{proof}

\begin{lemma}\label{lem:Dapprox}
Suppose $D$ is the Cartan masa of the hyperfinite II$_1$--factor $R$ and $a=a^*\in D$ has trace zero.
Let $n_1\in\Nats$, $n_2=n_3=\cdots=2$, and suppose
there exists an increasing family
\[
D^{(1)}\subseteq D^{(2)}\subseteq D^{(3)}\subseteq\cdots
\]
of abelian, unital $*$--subalgebras of $D$ whose union is weakly dense in $D$, and where each $D^{(j)}$ has
dimension $n_1n_2\cdots n_j$ and has minimal projections equally weighted by the trace.
Letting $E_j:D\to D^{(j)}$ denote the trace--preserving conditional expectation, (with $E_0$ being simply the trace),
suppose
\begin{equation}\label{eq:anjsum}
S:=\sum_{j=1}^\infty\|E_j(a)-E_{j-1}(a)\|<\infty.
\end{equation}
Then there is an automorphism $\sigma$ of $D$ and quasinilpotent element $z\in R$ so that $z^*+z=\sigma(a)$ and $\|z\|\le CS$,
where the constant $C$ is from Proposition~\ref{prop:Dm2}.
\end{lemma}
\begin{proof}
We may
write $D$ as an inductive limit as in the bottom row of~\eqref{eq:DinRindlim}, where we now think of $D^{(j)}$
as the set of diagonal matrices in $M_{2^j}$ and the inclusion $D^{(j)}\subseteq D^{(j+1)}$ given by
the map $\beta$ as in~\eqref{eq:betaeii}, with $m=n_1 2^{j-1}$ and $n=2$.
Then using that $D$ is a Cartan masa in $R$,
the inclusion $D\hookrightarrow R$ may be written as an inductive limit as in~\eqref{eq:DinRindlim}.
Let $a_j=E_j(a)-E_{j-1}(a)$.
Note that $E_0(a)=0$ and we have $a=\sum_{j=1}^\infty a_j$, with the estimate~\eqref{eq:anjsum} ensuring convergence in norm.
By Proposition~\ref{prop:Dm2}, for each $j$
there is a trace--preserving automorphism $\sigma_j$ of $D^{(j)}$ fixing each element of $D^{(j-1)}$
(if $j\ge2$) and such that
\[
\|T_{\sigma_j(a_j)}\|\le C\|a_j\|.
\]
The inductive limit of these automorphisms $\sigma_j$ is an automorphism $\sigma$ of $D$, and we have
\[
\sigma(a)=\sum_{j=1}^\infty\sigma_j(a_j).
\]
Now by Lemma~\ref{lem:Tasum}, there is a quasinilpotent element $z\in R$ such that $z+z^*=\sigma(a)$
and $\|z\|\le CS$.
\end{proof}

We now provide examples of the elements $a$ satisfying hypotheses of Lemma~\ref{lem:Dapprox}.

\begin{prop}\label{prop:aUHF}
 Let $D$ be the Cartan masa of the hyperfinite II$_1$--factor $R$, and let $a=a^*$ be an element in $D$ with $\tau(a)=0$.
Suppose, in addition, that  the distribution $\mu_a$ of $a$ satisfies:
\begin{enumerate}[(i)]
\item $\mu_a$ has at most a finite number of atoms, each of rational weight,
\item the nonatomic part of $\mu_a$ is either zero or has support equal to the union of finitely many pairwise disjoint
closed intervals $I_j$,
\item $\mu(I_j)$ is rational for each $j$,
\item the restriction of the nonatomic part of $\mu_a$ to each of the intervals $I_j$ is Lebesgue absolutely continuous and has
Radon--Nikodym derivative with respect to Lebesgue measure that is bounded below on $I_j$ by some $\delta>0$.
\end{enumerate}
Then the element $a$ satisfies the hypothesis of
Lemma~\ref{lem:Dapprox}
and, 
consequently, $\sigma(a)$ is the real part of a quasinilpotent operator in $R$, for some automorphism $\sigma$ of $D$.
\end{prop}
\begin{proof}
Let $n_1$ be an integer large enough so that (a) $n_1$ times the weight of every atom of $\mu_a$ is an integer and (b)
$n_1$ times each $\mu_a(I_j)$ is an integer.
Then we may choose an $n_1$ dimensional subalgebra $D^{(1)}$ of $D$ with minimal projections $p_1,\ldots,p_{n_1}$ that
are equally weighted
by the trace, and such that each $p_ja$ is either a scalar multiple of $p_j$ or an element whose distribution is
Lebesgue absolutely continuous on an interval with Radon--Nikodym derivative that is bounded below by $\delta$
on its support.

Now it suffices to consider a single element $b\in D$ whose distribution $\mu_b$ is Lebesgue
absolutly continuous, is supported on a closed interval $[c,d]$,
with $c<d$, and having Radon--Nikodym derivative with respect to Lebesgue measure that is bounded below by $\delta>0$.
It will suffice to find an increasing chain of
subalgebras $D^{(j)}$ of dimension $2^j$
and with all minimal projections having trace $2^j$, such that $\sum_{j=1}^\infty\|b-E_j(b)\|<\infty$,
where $E_j$ is the conditional expectation onto $D^{(j)}$.
This is easily done.
Indeed we have the partition
$c=c^{(j)}_0<c^{(j)}_1<\cdots<c^{(j)}_{2^j}=d$ of $[c,d]$ so that $\mu_b([c^{(j)}_{k-1},c^{(j)}_k])=2^{-j}$ for all $k$.
As the Radon--Nikodym derivative is bounded below by $\delta$, we have $c^{(j)}_k-c^{(j)}_{k-1}<2^{-j}/\delta$ for all $k$.
Then letting $D^{(j)}$ be the subalgebra of $D$ spanned by the spectral projections of $b$ corresponding to the
intervals $[c^{(j)}_{k-1},c^{(j)}_k]$, we have $\|b-E_j(b)\|\le2^{-j}/\delta$ and $D^{(j)}\subseteq D^{(j+1)}$.
\end{proof}

The techniques we have employed suggest the following question:
\begin{ques}
If $a$ is a self--adjoint element in the UHF algebra $M_{2^\infty}$ whose trace is zero, 
is $a$ the real part of a quasi--nilpotent operator?
\end{ques}

However, the key point for the previous proposition was to arrange that the series in~\eqref{eq:anjsum} be summable.
We do not see how to make this so for an abitrary element of the diagonal of the UHF algebra $M_{2^\infty}$ embedded in $R$.
The following example illustrates the difficulty.
\begin{example}
Let the Cartan masa $D$ be identified with $L^\infty[-\frac12,\frac12]$ with the trace given by Lebesgue measure.
Let $a\in D$ be the increasing function whose distribution $\mu_a$ is
\[
\mu_a=\sum_{n=2}^\infty\frac1{2^n}\big(\delta_{-\frac1n}+\delta_{\frac1n}\big).
\]
Thus, we have
\[
a(t)=\begin{cases}
-\frac1n,&-2^{-(n-1)}<t<-2^{-n},\quad n\ge2 \\
\frac1n,&2^{-n}<t<2^{-(n-1)},\quad n\ge2.
\end{cases}
\]
Let $D^{(n)}$ be the subalgebra of $D$ that is spanned by the characteristic functions of the intervals
\[
\big(-\frac12+\frac{k-1}{2^n},-\frac12+\frac k{2^n}\big),\qquad (1\le k\le 2^n)
\]
and let $E_n$ denote the conditional expectation of $D$ onto $D^{(n)}$.
Let $s_N=\sum_{n=N+1}^\infty\frac1{n2^n}$.
Then we have
\[
E_N(a)(t)=\begin{cases}
-\frac1n,&-2^{-(n-1)}<t<-2^{-n},\quad 2\le n\le N \\
-s_N,&-2^{-N}<t<0 \\
s_N,&0<t<2^{-N} \\
\frac1n,&2^{-n}<t<2^{-(n-1)},\quad 2\le n\le N.
\end{cases}
\]
and from this we compute
$\|E_{N+1}(a)-E_N(a)\|=\max(s_N-s_{N+1},|s_N-\frac1{N+1}|)$
for each $N\ge1$.
Since $0<s_N<2^{-N-1}$, we have
\begin{equation}\label{eq:suminfinite}
\sum_{N=1}^\infty\|E_{N+1}(a)-E_N(a)\|=\infty.
\end{equation}
\end{example}

While the above example does not prove that no choice of subalgebras $D^{(n)}$ can be made
which renders finite the corresponding sum~\eqref{eq:suminfinite},
we do not see a choice that would do so.

\medskip
The next proposition employs the usual techniques to give more examples in ultrapower II$_1$--factors.

\begin{prop}
Let $R^\omega$ be an ultrapower of the hyperfinite II$_1$ factor, for $\omega$ a non--principle ultrafilter on $\Nats$.
Let $a=a^*\in R^\omega$ have trace zero and suppose the distribution $\mu_a$ of $a$ satisfies the hypotheses of Proposition~\ref{prop:aUHF}.
Then there is a quasinilpotent element $z\in R^\omega$ such that $a=z+z^*$.
\end{prop}
\begin{proof}
By Proposition~\ref{prop:aUHF}, there is a quasinilpotent element $\zt\in R$ such that the distribution of $\at:=\zt+\zt^*$ equals $\mu_a$.
Thus, the element $b$ of $R^\omega$ which is the class of the sequence of $\at$ repeated infinitely often is (a) equal to $y+y^*$
for a quasinilpotent element $y$ of $R^\omega$ and (b) has distribution equal to $\mu_a$.
Since all the self--adjoint elements in $R^\omega$ having a given distribution are unitarily equivalent, we find the desired element $z$
as a unitary conjugate of $y$.
\end{proof}

Finally, here is a specific question 

\begin{ques}\label{ques:irrat}
Let $p$ be a projection in the hyperfinite II$_1$--factor or, for that matter, in any specific II$_1$--factor, whose trace $\tau(p)$ is irrational.
Is $p-\tau(p)1$ the real part of a quasinilpotent element of the II$_1$--factor?
\end{ques}
Of course, with $\tau(p)$ rational, the element $p-\tau(p)1$ is the real part of a nilpotent in an
embedded matrix algebra.
However, with $\tau(p)$ irrational, the techniques used in this paper do not apply to
the element $p-\tau(p)1$,
as it does not have the same distribution as any element
with vanishing center--valued trace in a finite type~I von Neumann algebra,
nor does it fall under the rubric of results in this section.

\end{document}